\documentclass[10pt,reqno]{article}

\pdfoutput = 1

\usepackage[english]{babel}
\usepackage{amsmath,amssymb,amsthm}
\usepackage{amsfonts}
\usepackage{amsxtra}
\usepackage{array,dcolumn}
\usepackage{graphicx}
\usepackage[hyperfootnotes=true,
        pdffitwindow=true,
        plainpages=false,
        pdfpagelabels=true,
        pdfpagemode=UseOutlines,
        pdfpagelayout=SinglePage,
        hyperindex,]{hyperref}

\usepackage[T1]{fontenc}
\usepackage[utf8]{inputenc}
\usepackage[activate={true,nocompatibility},spacing,kerning]{microtype}
\usepackage[sc,osf]{mathpazo}
\microtypecontext{spacing=nonfrench}

\newcommand{\mathsym}[1]{{}}
\newcommand{\unicode}[1]{{}}

\theoremstyle{plain}
\newtheorem{theorem}{Theorem}

\newtheorem{corollary}[theorem]{Corollary}
\newtheorem{proposition}[theorem]{Proposition}

\theoremstyle{definition}

\theoremstyle{remark}
\newtheorem{remark}[theorem]{Remark}

\begin{document}

\begin{center}
{\bfseries\Large Lyapunov exponents for some isotropic random matrix ensembles}
\\[2\baselineskip]
P. J. Forrester\footnote{pjforr@unimelb.edu.au} and Jiyuan Zhang\footnote{jiyuanz@student.unimelb.edu.au}%

{\itshape ARC Centre of Excellence for Mathematical and Statistical Frontiers,\\
School of Mathematics and Statistics, The University of Melbourne, Victoria 3010, Australia.}\\

\end{center}

\begin{abstract}
\noindent
A random matrix with rows distributed as a function of their length is said to be isotropic.
When these distributions are Gaussian, beta type I, or beta type II, previous work has,
from the viewpoint of integral geometry, obtained the explicit form of the distribution of the
determinant. We use these result to evaluate the sum of the Lyapunov spectrum of the 
corresponding random matrix product, and we further give explicit expressions for the
largest Lyapunov exponent. Generalisations to the case of complex or quaternion entries
are also given. For standard Gaussian matrices $X$, the full Lyapunov
spectrum for products of random matrices $I_N + {1 \over c} X$ is computed
in terms of a generalised hypergeometric function in general, and in terms of a single
single integral involving a modified Bessel function for the largest Lyapunov exponent.
\end{abstract}

\section{Introduction}\label{s1}

Let $X$ be a real $N \times N$ random matrix. A basic question is to ask about the distribution of
its determinant. Since $|\det X| = (\det X^T X)^{1/2}$, the determinant of the positive definite matrix
$X^T X$ provides a natural extension of this question to the case of $n \times N$ $(n \ge N)$ rectangular
matrices. In the latter setting, thinking of $X$ as a multi-dimensional data matrix, with the
measured quantities corresponding to the columns and each with mean zero, $X^T X$ is up to
proportionality the empirical correlation matrix between these quantities. Moreover, 
$(\det X^T X)^{1/2}$ has the geometrical interpretation as the Hausdorff volume of the parallelotope
generated by the columns of $X$; see e.g.~\cite{Ma99}.

Let $\{\tau_j\}_{j=1}^N$ denote the singular values of $X$. Since $| \det X | = \prod_{l=1}^N \tau_l$, one
approach to this question is to make use of knowledge of the distribution of the singular values.
For example, in the case that $X$ is an $N \times N$ standard Gaussian matrix, it is known
(see e.g.~\cite[Prop.~3.22 with $\beta = 1$, $n=m=N$]{Fo10}) that the joint probability density function
(PDF) for the squared singular values $\lambda_l = \tau_l^2$ is proportional to
\begin{equation}\label{1}
\prod_{l=1}^N \lambda_l^{-1/2} e^{- \lambda_l/2} \chi_{\lambda_l > 0}
\prod_{1 \le j < k \le N} | \lambda_k - \lambda_j |,
\end{equation}
where $\chi_J = 1$ for $J$ true and $\chi_J = 0$ otherwise. Studying the distribution of 
$|\det X|^2 = \prod_{l=1}^N \lambda_l$ through its moments then calls for the evaluation of
\begin{equation}\label{2}
\int_0^\infty d \lambda_1 \cdots \int_0^\infty d \lambda_N \,
\prod_{l=1}^N \lambda_l^{-1/2 + s} e^{-\lambda_l/2} \prod_{1 \le j < k \le N} | \lambda_k - \lambda_j|.
\end{equation}
This is a particular (limiting) Selberg integral (see e.g.~\cite[\S 4.7]{Fo10}), and as such can be
evaluated as a product of gamma functions to give
\begin{equation}\label{2.1a}
\Big \langle \prod_{l=1}^N \lambda_l^s \Big \rangle_X =
\prod_{j=1}^N {2^s \Gamma(s+j/2) \over \Gamma(j/2)}.
\end{equation}
With $\chi_j^2$ denoting the chi-squared distribution with $j$ degrees of freedom, we read off from
this that
$$
\Big \langle \prod_{l=1}^N \lambda_l^s \Big \rangle_X = \prod_{j=1}^N 
\Big \langle  \lambda_j^s \Big \rangle_{\chi_j^2}
$$
and hence (see e.g.~\cite{Mu82})
\begin{equation}\label{3}
|\det X|^2 \mathop{=}\limits^{\rm d} \prod_{j=1}^N \chi_j^2.
\end{equation}

Our focus in the present paper relates to a different viewpoint on (\ref{3}), which in fact can
be traced back to the work of Bartlett \cite{Ba34}. Instead of decomposing $X$ in terms of its
singular values, the strategy is to use the QR (Gram-Schmidt) decomposition in which $Q$ is a real
orthogonal matrix consisting of the Gram-Schmidt basis as its columns, and $R=[r_{jk}]_{j,k=1}^N$
is an upper triangular matrix with positive entries on the diagonal. With $(dY)$ denoting the
product of differentials of the matrix $Y$, the change of variables formula (see e.g.~\cite{Mu82})
\begin{equation}\label{3.1}
(dY) = \prod_{l=1}^N r_{ll}^{N-l} (d R) (Q^T d Q),
\end{equation}
and the facts that $\prod_{j,k=1}^N e^{- x_{jk}^2/2} = \prod_{j \le k}^N
e^{- r_{jk}^2/2}$, imply that each variable $r_{jj}^2$ has distribution $\chi_{N-j+1}^2$. Noting that
\begin{equation}\label{3.1a}
\det X^T X = \prod_{j=1}^N r_{jj}^2
\end{equation}
 then reclaims (\ref{3}).

The technique of 
derivation of (\ref{3}) using the QR rather than singular value decomposition is a more powerful
method, Thus it provides for the exact specification of the distribution of $| \det X |^2$ for
some ensembles of matrices $X$ with distribution invariant under the transformation
\begin{equation}\label{4}
X \mapsto X Q,
\end{equation}
for $Q$ a real orthogonal matrix, but otherwise not a function of $X^T X$. The latter point means 
in particular that such ensembles do not permit
a formula analogous to (\ref{1}) for the PDF of the squared singular values. An example  is the so
called uniform Gram ensemble, in which each row of $X$ is chosen uniformly from the
unit $(N-1)$-sphere. Use of the QR decomposition gives that  \cite{Ro07,TD06}
\begin{equation}\label{5}
| \det X |^2  \mathop{=}\limits^{\rm d} \prod_{l=1}^N  {\rm Beta} \, [(N-l+1)/2,(l-1)/2],
\end{equation}
where ${\rm Beta} \, [\alpha, \beta]$ denotes the beta distribution with parameters $(\alpha, \beta)$.

In Section 2 of this paper results on volumes of random parallelotopes in the integral geometry
literature will be used to specify random matrix ensembles with the invariance (\ref{4}) which
permit results analagous to (\ref{3}) and (\ref{5}). 
The entries of each row of $X$ are drawn from a Gaussian, beta type I, or beta type II distribution dependent only on the length of the row.
We show that these results can be extended to
the case that the matrices in the ensembles have complex, or quaternion, entries rather than being
real valued as in the integral geometry setting. There is an application of these results
to computations relating to the Lyapunov spectrum $\{\mu_j\}_{j=1}^N$ for the random
product matrix
\begin{equation}\label{PX}
P_m = X_m X_{m-1} \cdots X_1,
\end{equation}
where each $X_i$ is chosen independently from an ensemble with the invariance (\ref{4}). The
main point underlying this is that the formula of Oseledec \cite{Os68} and
Raghunathan \cite{Ra79} for the sum of the first $k$ largest exponents
\begin{equation}\label{5.1}
\mu_1 + \cdots + \mu_k = \sup \lim_{m \to \infty} {1 \over m}  \log {\rm Vol}_k \{y_1(m),\dots,y_k(m) \} \qquad (k=1,\dots,d),
\end{equation}
where $y_j(m) := P_m y_j(0)$, the sup operation is over all sets of linearly independent vectors
$\{y_1(0),\dots,y_k(0) \}$ and Vol${}_k$ refers to the volume of the parallelotope generated by the given set of $k$ vectors, simplifies in the setting of the invariance (\ref{4}). Thus, then \cite{CN84,Ne86}
\begin{equation}\label{6}
\mu_1 + \cdots + \mu_k =   \Big \langle \log \det  \Big ( X_{N \times k}^T  X_{N \times k}\Big )^{1/2} \Big 
\rangle,
\end{equation}
where $X_{N \times k}$ is the restriction of $X$ to its first $k$ columns. The results of Section 2 apply directly to the case $k=N$ of (\ref{6}). An analogous study can be carried out for the products
of independent random matrices of the form
\begin{equation}\label{IX}
I_N + {1 \over c} X,
\end{equation}
where $X$ is a standard Gaussian matrix. Although the resulting expression
is computationally more difficult, there is simplification in the limit $c \to \infty$.

In Section 3 we address the problem of using (\ref{IX}) to compute the largest Lyapunov exponent
$\mu_1$ for the ensembles
of Section 2, and give evaluations in terms of a single integral. It is well known
 that for products of standard Gaussian random matrices the RHS of (\ref{5.1}) can
be computed explicitly \cite{Ne86}. We show in Section 4 that for product ensembles based
on (\ref{IX}), an evaluation of (\ref{5.1}) in terms of a generalised hypergeometric function can be
given. For the largest Lyapunov exponent, a much
simpler expression, involving only a single integral of standard special functions, is also obtained.

\setcounter{equation}{0}
\section{Distribution of the determinant for some invariant ensembles}\label{S2}
\subsection{Isotropic rows}
Let the row vector $\mathbf x = (x_1,\dots, x_N)$ be sampled from a distribution with one of the PDFs
\begin{equation}\label{2.1}
\left \{
\begin{array}{ll} (2 \pi \sigma^2)^{-N/2} e^{- | \mathbf x|^2/2 \sigma^2}, \: \mathbf x \in \mathbb R^N
& ({\rm Gaussian}) \\
\pi^{-N/2} (\nu/2)_{N/2} (1 - |\mathbf x|^2 )^{\nu/2 - 1},  \: \mathbf x \in B_N &
({\rm Beta \: type \: I}) \\
\pi^{-N/2} ( \omega /2)_{N/2} (1 + |\mathbf x|^2 )^{-(N+\omega)/2}, \: \mathbf x \in  \mathbb R^N &
({\rm Beta \: type \: II}).
\end{array} \right.
\end{equation}
Here $(c)_n = \Gamma(c+n)/\Gamma(c)$ and $B_N$ denotes the unit ball in
$\mathbb R^N$. We require $\sigma > 0$ in the Gaussian
weight, $\nu > 0$ in the beta type I weight (the limit $\nu \to 0^+$ corresponds to the
uniform Gram ensemble) and $\omega \ge 0$ in the beta type II weight.
Being functions of $|\mathbf x |^2$, each of these has the vector invariance corresponding to
(\ref{4}), $\mathbf x \mapsto \mathbf x R$. In the terminology of \cite{Ru79} the distribution is then
referred to as isotropic.

The Gaussian case is realised by choosing each $x_i$ in $\mathbf x$ to be an independent normal
random variable with standard deviation $\sigma$. For $\nu \in \mathbb N_{0}$, the distribution corresponding
to the beta type I weight can be sampled by choosing a vector uniformly at random on the unit
sphere in $\mathbb R^{N+\nu}$, then restricting to the first $N$ components; this follows from
e.g.~\cite[Eq.~(3.113)]{Fo10}. For $\omega > 0$, the 
distribution corresponding
to the beta type II weight can be sampled by choosing $w  \mathop{=}\limits^{\rm d}
\chi_\omega^2$, and then forming $(1/\sqrt{w})\mathbf v$ where $\mathbf v$ is a
$1 \times N$ row vector of standard Gaussians \cite{Ru79}.

We form three random matrix ensembles by choosing rows independently according to 
distributions corresponding to the three
weights in (\ref{2.1}), allowing for the parameters $\sigma, \nu$ and $\omega$ to be different in the
different rows. 
We remark that in the Gaussian case the random matrix can equivalently be specified as the
product
\begin{equation}\label{GK}
\Sigma^{1/2} G,
\end{equation}
where $\Sigma$ is the diagonal matrix ${\rm diag} \, [\sigma_1^2,\dots, \sigma_N^2]$, and $G$ has
independently distributed standard Gaussian elements.
Results in \cite{Ru79} give the distribution of the determinant and its moments.
In addition to the $\chi^2$ and beta distributions, also required is the beta prime distribution,
to be denoted BetaP[a,b], which is supported on $(0,\infty)$ and has density proportional to
$y^{a-1}/(1+y)^{a+b}$ (see \cite{FI18} for another recent application of the latter in random matrix
theory).

\begin{proposition}\label{R} (Ruben)
For the Gaussian weight \cite[Eqs.~(29), (28)]{Ru79}
\begin{equation}\label{2.2}
\Big \langle | \det X |^{2 \alpha} \Big \rangle =
\prod_{l=1}^N (2 \sigma_l^2)^\alpha \, ((N-l+1)/2)_\alpha,
\end{equation}
and
\begin{equation}
 |\det X |^{2} \mathop{=}\limits^{\rm d}  \prod_{l=1}^N \sigma_l^2
 \chi_l^2
\end{equation}
(cf.~(\ref{2.1}) and (\ref{3});
for the beta type I weight \cite[Eqs.~(36), (35)]{Ru79}
\begin{equation}\label{36}
\Big \langle | \det X |^{2 \alpha} \Big \rangle =
\prod_{l=1}^N {((N-l+1)/2)_\alpha \over ((N + \nu_l)/2)_\alpha}
\end{equation}
and
\begin{equation}
| \det X |^{2 } \mathop{=}\limits^{\rm d}  \prod_{l=1}^N {\rm Beta}[
 (N-l+1)/2,  (\nu_l + l - 1)/2];
\end{equation}
for the beta type II weight \cite[Eqs.~(34), (33)]{Ru79}
\begin{equation}
\langle | \det X |^{2 \alpha} \rangle =
\prod_{l=1}^N ((N-l+1)/2)_\alpha (\omega_l/2)_{-\alpha}
\end{equation}
and
\begin{equation}
 |\det X |^{2 } \mathop{=}\limits^{\rm d}  \prod_{l=1}^N {\rm BetaP}[
  (N-l+1)/2, \omega_l/2].
\end{equation}
\end{proposition}

Note that the case $\nu_l = 0$ ($l=1,\dots,N$) of (\ref{36}) reclaims (\ref{6}).
Also, we remark that random polytopes formed by the convex hull of the columns of
$X^T$, in the case that $X$ is of size $n \times N$ $(n > N)$ with rows chosen according to
one of the distributions (\ref{2.1}), have been the subject of the recent studies \cite{KTT17,KMTT18}.

The method of proof of Proposition \ref{R},
implied by working due to Mathai \cite{Ma99a} formulated two decades after \cite{Ru79}, 
relies on the QR (Gram-Schmidt) decomposition and the corresponding
change of variables formula (\ref{3.1}). This strategy can be generalised to the case that the matrix $X$
has complex entries, or has quaternion entries, with rows distributed according to one of the
three specifications in (\ref{2.1}).  There are two real components in a complex number and
four in a quaternion. This requires that $N$ in (\ref{2.1}) be replaced by $\beta N$, where 
$\beta = 1,2,4$ for real, complex, quaternion entries respectively, and also in the beta type I
case that $\nu$ be replaced by $\beta \nu$.

In the case of complex entries, $Q$ is complex unitary
and will be denoted $U$.
The matrix $R = [r_{jk} ]_{j,k=1}^N$ is again upper triangular with positive entries on the diagonal, but now
the strictly upper triangular entries are complex. The corresponding
change of variables formula reads (see e.g.~\cite[Prop.~3.2.5]{Fo10})
\begin{equation}\label{3.2}
(dX) = \prod_{j=1}^N r_{jj}^{2(N-j) + 1} (d T) (U^\dagger d U),
\end{equation}
and analogous to (\ref{3.1a})
\begin{equation}\label{3.3a}
\det X^\dagger X = \prod_{j=1}^N r_{jj}^2.
\end{equation}

We say a $2N \times 2N$ matrix $X$ has quaternion entries if it consists entirely of $2 \times 2$ blocks
of the form 
\begin{equation}\label{3.2b}
\begin{bmatrix} z & w \\ - \bar{w} & \bar{z} \end{bmatrix}. 
\end{equation}
In the QR (Gram-Schmidt) decomposition the matrix $Q$ is complex unitary with quaternion entries, and is to be denoted by $U$. The matrix $R$ is now block
upper triangular, with diagonal entries $r_{jj} \begin{bmatrix} 1 & 0 \\  0 & 1 \end{bmatrix}$, $r_{jj} > 0$,
(a positive scalar quaternion) and the strictly upper triangular entries are quaternions. For the change of variables formula
we have (\cite[Exercises 3.2 q.5(i)]{Fo10}, \cite[Lemma 2.7]{DG09}
\begin{equation}\label{3.3}
(dX) = \prod_{j=1}^N r_{jj}^{4(N-j) + 3} (d R) (U^\dagger d U).
\end{equation}
Note that as a $2N \times 2N$ complex matrix, $X$ would satisfy (\ref{3.3a}) with the RHS
squared.  However, regarding $X$
as a quaternion matrix the appropriate definition (see e.g.~\cite[Def.~6.1.2]{Fo10}) gives that	
(\ref{3.3a}) is again to be used as specifying the determinant.

In terms of the notation $\beta$ introduced in the paragraph above (\ref{3.2}),
we remark that (\ref{3.1}), (\ref{3.2}) and (\ref{3.3}) can be summarised by the one expression
\begin{equation}\label{3.4}
(d X )  = \prod_{j=1}^N r_{jj}^{\beta(N - j + 1) - 1} (d R) (U^\dagger dU).
\end{equation}
Moreover, upon applying the QR (Gram-Schmidt) decomposition to $X^\dagger$ we see that
\begin{equation}\label{rX}
(X X^\dagger)_{jj} = | \mathbf x _j|^2 = \sum_{l=1}^j  r_{lj}^2,
\end{equation}
where in the quaternion case the notation $( \cdot )_{jj}$ refers to the scalar multiple of the
identity in position $(jj)$ of the diagonal and $\mathbf x _j$ refers to the $j$-th row of $X$ regarded
as a matrix of quaternions. On this latter point note that with $w=u+iv$ and $z=x+iy$ in
(\ref{3.2b}), the modulus squared of the corresponding quaternion is equal to $u^2+v^2+x^2+y^2$.
This is in keeping with the presentation of a quaternion in terms of its units $\{1,i,j,k\}$ as
a linear combination with real components, $u + vi + xj + yk$.

Suppose that row $j$ of the matrix $X$ has a distribution with PDF $f_j( | \mathbf x_j |^2)$.
Then it follows from (\ref{3.4}) and (\ref{rX}) that the joint PDF of $\{r_{jj} \}$ is
proportional to
\begin{equation}\label{3.5a}
\prod_{j=1}^N r_{jj}^{\beta(N - j + 1) - 1} \int_{\{ r_{jl} \}_{l=1}^{j-1}}
f_j\Big ( \sum_{l=1}^j r_{lj}^2 \Big ) \, 
\prod_{l=1 }^{j - 1} dr_{jl}.
\end{equation}
Note that $d r_{jl} = \prod_{s=1}^\beta dr_{jl}^{(s)}$, where
$dr_{jl}^{(s)}$
are the differentials associated the independent real numbers
which make up $r_{jl}$.
For the particular $f_j$ as implied by (\ref{2.1}), (\ref{3.5a}) has the property that its dependence
on $\{ r_{jj} \}$ factorises in each of the real, complex and quaternion cases. This allows
Proposition \ref{R} to be correspondingly generalised.

\begin{proposition}\label{R1} 
Consider the random matrix ensembles with rows distributed according to
(\ref{2.1}), appropriately generalised as noted in the paragraph above (\ref{3.2})
to allow for complex or quaternion entries. Use the notation $\beta$ of that
paragraph to distinguish the number field.
For the Gaussian weight 
\begin{equation}\label{r1.1}
\Big \langle | \det X |^{2 \alpha} \Big \rangle =
\prod_{l=1}^N (2 \sigma_l^2)^\alpha \, (\beta (N-l+1)/2)_\alpha,
\end{equation}
and
\begin{equation}\label{r1.2}
 |\det X |^{2} \mathop{=}\limits^{\rm d}  \prod_{l=1}^N \sigma_{l}^2
 \chi_{\beta l}^2;
\end{equation}
for the beta type I weight 
\begin{equation}\label{r1.3}
\Big \langle | \det X |^{2 \alpha} \Big \rangle =
\prod_{l=1}^N {(\beta (N-l+1)/2)_\alpha \over (\beta (N + \nu_l)/2)_\alpha}
\end{equation}
and
\begin{equation}\label{r1.4}
| \det X |^{2 } \mathop{=}\limits^{\rm d}  \prod_{l=1}^N {\rm Beta}[
 \beta (N-l+1)/2,  \beta (\nu_l + l - 1)/2];
\end{equation}
for the beta type II weight 
\begin{equation}\label{r1.5}
\Big \langle | \det X |^{2 \alpha} \Big \rangle =
\prod_{l=1}^N (\beta (N-l+1)/2)_\alpha (\beta \omega_l/2)_{-\alpha}
\end{equation}
and
\begin{equation}\label{r1.6}
 |\det X |^{2 } \mathop{=}\limits^{\rm d}  \prod_{l=1}^N {\rm BetaP}[
  \beta (N-l+1)/2, \beta \omega_l/2].
\end{equation}
\end{proposition}

\begin{proof}
In the Gaussian case $f_j(u)$ in (\ref{3.5a}) is proportional to $e^{-u/2 \sigma_j^2}$. The dependence
in the integrand on $r_{ll}$ thus factorises, showing immediately that the joint distribution of
$\{r_{jj} \}$ is proportional to 
$$
\prod_{j=1}^N r_{jj}^{\beta(N-j+1) - 1} e^{- r_{jj}^2/2 \sigma_j^2}.
$$
Recalling (\ref{3.3a}), the results (\ref{r1.1}) and (\ref{r1.2}) follow as a corollary.

In the case of the beta type I weight, we have that $f_j(u)$ is proportional to
$(1 - u)^{\nu_j/2 - 1}$, $0<u<1$. The dependence on $r_{jj}$ is determined by the change of
variables $r_{lj} \mapsto (1 - r_{jj}^2)^{1/2} r_{lj}$ $(l < j )$, with $0 < r_{jj} < 1$ in (\ref{3.5a}).
This shows that the joint PDF of $\{r_{jj} \}$ is proportional to
$$
\prod_{j=1}^N r_{jj}^{\beta (N-j+1)-1} (1 - r_{jj}^2)^{\beta(j-1+\nu_j)/2 - 1}.
$$
Upon recalling (\ref{3.3a}), we deduce from this both (\ref{r1.3}) and (\ref{r1.4}).

The workings needed to deduce (\ref{r1.5}) and (\ref{r1.6}) for the beta type II weight, proceeds
in an analogous way to the beta type I case, except now the change of variables is
$r_{lj}  \mapsto (1 + r_{jj}^2)^{1/2} r_{lj}$ $(l < j )$.

\end{proof}

\begin{remark}
The above working applies simultaneously to cases that the entries of the random matrix are
real, complex or quaternion. In the real case ($\beta = 1$) this then provides a proof of
Ruben's result Proposition \ref{R}. Ruben's proof \cite{Ru79} is seemingly quite different, making use of
(amongst other geometrical/ statistical constructions) the exact form of the distribution of an isotropic
random vector in $\mathbb R^n$ projected orthogonally onto a lower dimensional subspace.

The moments for the Gaussian weight for all three number systems, and similarly the moments
for the beta type I weight with $\nu_l=0$ (uniform on the sphere) and $\beta \nu_l/2 = 1$ (uniform 
on the ball), have been computed in the recent work \cite[Prop.~8]{Fo17}; see also \cite{Mo12}
in the Gaussian case. The working relies on (Miles version of) the 
 Blaschke-Petkantschin formula from integral geometry.
\end{remark}

\begin{remark}\label{R4}
As already mentioned below (\ref{2.1}), a row distribution obtained from the beta type I weight with $\nu \to 0^+$ and $\beta = 1$
gives the Gram ensemble in which rows are sampled uniformly from the surface of the unit sphere
in $\mathbb R^N$. As noted in \cite{Pr67}, the latter sampling can be done by normalising a
standard Gaussian vector, which has the consequence that  for $Y$ a member of the Gram ensemble and
$X$ a standard Gaussian matrix
$$
\prod_{l=1}^N \chi_{\beta N}^2 \, | \det Y |^2 
\mathop{=}\limits^{\rm d} | \det X|^2.
$$
The moment formulas (\ref{r1.1}) (with $\sigma_l=1/2$) and (\ref{r1.3}) (with $\nu_l=0$) are
consistent with this relation.

We remark too that the construction of vectors with distributions (\ref{2.1}), as given in the paragraph
subsequent to (\ref{2.1}), extend naturally to the setting of complex and quaternion
entries. Thus in the Gaussian case, the independent part of each entry is to be a normal random variable with mean zero and standard deviation $\sigma$. For the beta type I weight with $\nu \in \mathbb N_0$,
one chooses a standard Gaussian vector with $N + \nu_0$ complex or quaternion entries,
normalises it to obtain a vector chosen uniformly from the unit sphere in $\mathbb C^{N + \nu_0}$
or $\mathbb H^{N+ \nu_0}$, and restricts to the first $N$ components;
see \cite[Eq.~(3.113)]{Fo10}.
Choosing a $1 \times N$ row
vector $\mathbf v$ with standard complex or quaternion entries, then forming
$(1/\sqrt{w}) \mathbf v$ with $w \mathop{=}\limits^{\rm d} \chi_\omega^2$
gives the sought generalisation of the beta type II weight.
\end{remark}

Of particular interest, in light of (\ref{6}) with $k=N$, and its generalisation to the complex
and quaternion cases, is the expected value of $\log | \det X|$.

\begin{corollary}\label{C4}
Use $\beta$ to distinguish the number field corresponding to the entries of $X$. Let
$\Psi(x) := {d \over dx} \log \Gamma(x)$ denote the digamma function. For the Gaussian weight
\begin{equation}\label{cw.1}
\Big \langle \log | \det X | \Big \rangle = {1 \over 2}
\sum_{l=1}^N \bigg ( \log 2 \sigma_l^2 + \Psi \Big ( \beta(N-l+1)/2 \Big ) \bigg );
\end{equation}
for the beta type I weight
\begin{equation}\label{cw.2}
\Big \langle \log | \det X | \Big \rangle = {1 \over 2}
\sum_{l=1}^N \bigg ( \Psi \Big ( \beta (N-l+1)/2 \Big ) - \Psi \Big ( \beta(N + \nu_l)/2 \Big ) \bigg );
\end{equation}
and for the beta type II weight
\begin{equation}\label{cw.3}
\Big \langle \log | \det X | \Big \rangle = {1 \over 2}
\sum_{l=1}^N \bigg ( \Psi \Big ( \beta (N-l+1)/2 \Big ) - \Psi \Big ( \beta \omega_l/2 \Big ) \bigg ).
\end{equation}
\end{corollary}

Specifically, due to each of the matrix ensembles considered in this section having the invariance
(\ref{4}), forming the random product matrix (\ref{PX}), we have from (\ref{6}) that the sum of the Lyapunov
exponents satisfy
\begin{equation}\label{2.x2}
\mu_1 + \cdots + \mu_N = \Big \langle \log | \det X | \Big \rangle
\end{equation}
and are thus given explicitly by the results of Corollary \ref{C4}.

\subsection{Noncentral Wishart matrices}\label{S2.2}
Consider the Gaussian case of (\ref{2.1}), and suppose $\sigma_i = \sigma$ for each row $i$
of the corresponding random matrix $X$.
Then the joint distribution of the elements is proportional to $e^{- {\rm Tr} \, X^T X/2 \sigma^2}$.
The latter exhibits an invariance  under the class of transformations 
$
X \mapsto R_2 X R_1,
$
where $R_1, R_2$ are real orthogonal matrices, extending the invariance (\ref{4}). Let $c$ be a scalar,
and introduce the random matrix $Y:= cI_N + X$, where $I_N$ denotes the identity.
 This is invariant under transformations of the form
\begin{equation}\label{4a}
Y \mapsto R^T Y R, 
\end{equation}
for $R$ real orthogonal. 
Forming $W = Y^T Y$ then gives an example of a non-central Wishart matrix; see 
e.g.~\cite{Mu82}.
According to Constantine \cite{Co63},
generalised to the complex case in \cite[Eq.~(21)]{CZ05}, with $\tilde{c} = c^2/2 \sigma^2$ we have
\begin{align}\label{FW}
\Big \langle (\det W)^\alpha \Big \rangle & =
\prod_{l=1}^N (2 \sigma^2)^\alpha (\beta (N-l+1)/2)_\alpha \,
e^{-N \tilde{c}} \, {}_1^{} F_1^{(2/\beta)}(\alpha + \beta N/2;\beta N/2;\tilde{c}I_N) \nonumber \\
& = \prod_{l=1}^N (2 \sigma^2)^\alpha (\beta(N-l+1)/2)_\alpha \,
{}_1^{} F_1^{(2/\beta)}(-\alpha,\beta N/2;-\tilde{c}I_N),
\end{align}
where the second equality follows from a generalization of one of Kummer's transformations;
see e.g.~\cite[Th.~7.4.3]{Mu82}, \cite[Eq.~(13.16)]{Fo10}. Here 
${}_1^{} F_1^{(2/\beta)}$ is, for $\beta = 1,2$ and 4, a hypergeometric function of matrix argument
$X$, dependent on its eigenvalues only (see e.g.~\cite[Ch.~13]{Fo10}.
Since the hypergeometric function is normalised to unity for $X$ equal to the zero matrix, we
see that (\ref{FW}) correctly reduces to (\ref{2.2}) in this special case.
On the other hand, for $c \ne 0$ the functional dependence on $\alpha$ is now much more
complicated.

When $\alpha$ is a positive integer, the series form of ${}_1^{} F_1^{(2/\beta)}$
in the second line of (\ref{FW}) terminates
 (see e.g.~\cite[eq.~(13.2)]{Fo10}), leaving a polynomial in $c$. This is
 consistent with the LHS of (\ref{FW}), as implied by the meaning of $W$. 
 For $\tilde{c}$ large, a known \cite[Th.~3.2]{CM76}, \cite[\S 7]{Mu78} asymptotic formula
 for ${}_1^{} F_1^{(2/\beta)}$ in the case $\beta = 1$, appropriately generalised to all
 $\beta > 0$ using \cite[Eqns.~(13.12), (13.4), (13.2)]{Fo10} gives
 \begin{equation}\label{2.26}
 \Big \langle \Big ( \det \Big | \mathbb I_N + {1 \over c} X \Big |  \Big )^{2 \alpha }   \Big \rangle \\
 \sim {}_2^{} F_0^{(2/\beta)} \Big ( 1 - \beta/2 - \alpha, - \alpha; {1 \over \tilde{c}}  \mathbb I_N\Big ).
 \end{equation}
 By differentiating with respect to $\alpha$ and setting $\alpha = 0$, this gives that to leading
 order
  \begin{equation}\label{2.28}
 \Big \langle
 \log \det \Big | \mathbb I_N + {1 \over c} X \Big |  \Big \rangle
 \mathop{\sim}\limits_{c \to \infty} {\sigma^2 N (\beta/2 -1) \over c^2}.
  \end{equation}
  Notice that this vanishes for $\beta = 2$. In fact in the case $\beta = 2$, 
  differentiating (\ref{2.26}) with respect to $\alpha$ and setting $\alpha = 0$ gives zero for
  each term in the asymptotic expansion. This indicates that then the large $c$ decay is exponentially
  fast. In the case $N=1$, the matrix hypergeometric function in (\ref{FW}) reduces to the
  classical confluent hypergeometric function, and the exponentially small term has been exhibited explicitly in \cite{Pa13}.
  
  In keeping with (\ref{2.x2}), updated so that the RHS is given by the LHS of (\ref{2.28}), the
  result (\ref{2.28}) gives the leading term in the large $c$ asymptotic expansion for the sum of
  the Lyapunov exponents of the random product matrix formed by $(\mathbb I_N + {1 \over c} X)$.

 %With $M_N$ denoting
% the Morris integral (see \cite[Eq.~(4.4)]{Fo10}), we know from \cite[Exercises 13.1 q.4]{Fo10}
% that
% \begin{multline*}
% {}_1^{} F^{(2/\beta)}_1(-b; a + 1 +  \beta (N - 1) / 2;
%(t)^N) 
%= {1 \over M_N(a,b,\beta/2)}  \\
%\times
%\int_{-1/2}^{1/2}dx_1 \cdots \int_{-1/2}^{1/2}dx_N \,
%\prod_{l=1}^N 
%e^{\pi i x_l(a-b)} |1 + e^{2 \pi i x_l}|^{a+b}
%e^{-t e^{2 \pi i x_l}} \prod_{j < k}
%|e^{2 \pi i x_k} - e^{2 \pi i x_j}|^{\beta}.
%\end{multline*}
%Substituting in (\ref{FW}) gives
% \begin{multline}
%\Big \langle (\det W)^\alpha \Big \rangle 
% = \prod_{l=1}^N (2 \sigma^2)^\alpha \, {(\Gamma(\beta/2))^N \over N!}
%\\
%\times
%\int_{-1/2}^{1/2}dx_1 \cdots \int_{-1/2}^{1/2}dx_N \,
%\prod_{l=1}^N 
%e^{\pi i x_l(a-\alpha)} |1 + e^{2 \pi i x_l}|^{a+\alpha}
%e^{ \tilde{c} e^{2 \pi i x_l}} \prod_{j < k}
%|e^{2 \pi i x_k} - e^{2 \pi i x_j}|^{\beta},
%\end{multline}
%where $a=-1+\beta/2$. Differentiating both sides with respect to $\alpha$ and setting
%$\alpha = 0$ gives an exact, albeit cumbersome relative to the exact results of Corollary 4, evaluation
%of $\langle | \log ( c I_N + X) | \rangle$
\setcounter{equation}{0}
\section{Largest Lyapunov exponent}
\subsection{General formalism}
Consider the three random matrix ensembles specified in terms of the row distributions (\ref{2.1}).
Extend each by allowing for either real, complex or quaternion entries. Now form the random product
matrix (\ref{PX}). Since each $X_i$ is invariant under right multiplication by a unitary matrix with elements
from the corresponding number field, we know from \cite{CN84,Ne86} that the Lyapunov
exponents satisfy (\ref{6}), with $X^T$ replaced by $X^\dagger$ in the complex and quaternion
cases. This is a dramatic simplification from (\ref{5.1}), or the characterisation from Oseledec's
multiplicative ergodic theorem \cite{Os68}, which tells us that the spectrum of
$\lim_{m \to \infty} (P_m^\dagger P_m)^{1/2m}$ is given by $\{ e^{\mu_k} \}_{k=1}^N$.

In the case $k=1$, (\ref{6}) involves the sum of the absolute values squared of the first column of $X$
(since the rows of $X$ are isotropic, this may as well be any one of the $N$ columns of $X$). 
Thus
\begin{equation}\label{m1}
\mu_1 = {1 \over 2} \int_0^\infty (\log t ) p(t) \, dt,
\end{equation}
where $p(t)$ is the PDF for $\sum_{l=1}^N |x_{l,1} |^2$, with
$x_{l,1}$ denoting the entries in the first column
of $X$.
To
compute $p(t)$, the first task is to deduce from the row distributions (\ref{2.1})
the marginal distribution of a single element.

\begin{proposition}
Consider a random vector $\mathbf x = (x_1,\dots, x_N)$, with entries that are either real $(\beta = 1)$,
complex ($\beta = 2$) or quaternion ($\beta = 4$). Let the PDF of the distribution of $\mathbf x$ be
given by one of the three functional forms in (\ref{2.1}), appropriately modified in the complex and
quaternion cases (recall the paragraph above (\ref{3.2})). Each entry $x_i$ then involves $\beta$
independent reals.
 In the Gaussian case, the PDF of the marginal distribution of a single entry $x$ of $\mathbf x$ is
\begin{equation}\label{dw.1}
 (2 \pi \sigma^2)^{-\beta/2} e^{-|x |^2/2 \sigma^2};
\end{equation}
for the beta type I weight it is given by
\begin{equation}\label{dw.2}
{\Gamma(\alpha + \beta/2) \over \pi^{\beta/2} \Gamma(\alpha) } (1 - |x|^2)^{\alpha - 1}
\Big |_{\alpha = \beta (N-1+ \nu )/2 }, \qquad
|x| < 1;
\end{equation}
and for the beta type II weight by
\begin{equation}\label{dw.3}
 { \Gamma( (\omega + \beta)/2 ) \over \pi^{\beta/2} \Gamma(\omega/2)} 
(1 + |x|^2)^{-{1 \over 2} (\omega + \beta)}.
\end{equation}
\end{proposition}

\begin{proof}
In the complex case an element is of the form $x^{(1)} + i x^{(2)}$ with $x^{(1)}, x^{(2)} \in \mathbb R$
and $|x|^2 = (x^{(1)})^2 +  (x^{(2)})^2$. In the quaternion case, written for present purposes in terms of
its units $\{1,i,j,k\}$, an element is of the form $x = x^{(1)} + i x^{(2)} + j x^{(3)} + k x^{(4)}$ with
$x^{(s)} \in \mathbb R$, $(s=1,\dots,4)$ and $|x|^2 = \sum_{s=1}^4 (x^{(s)})^2$. By the fact that the
PDFs (\ref{2.1}) are functions only of $|\mathbf x |^2$, for the given number system each 
component $x$ has the same distribution. The task to compute this distribution is to integrate the
respective PDFs over all components but one. Due to the factorisation property of the Gaussian,
the result (\ref{dw.1}) then follows immediately for this weight.

Consider next the beta type I weight. Suppose for definiteness that the component of interest
is $x_1$. For each of the  other components
$\{x_j\}_{j=2,,\dots,N}$ we change variables $x_j \mapsto (1 - 
|x_1|^2)^{1/2} x_j$.
 This shows that the marginal PDF of $x_1$ is proportional to
\begin{equation}\label{CN}
(1 - |x|^2)^{\beta (N - 1 + \nu )/2 -1}, \qquad |x| < 1,
\end{equation}
thus implying (\ref{dw.2}).

The analogous change of variables for the beta type II weight gives that the marginal PDF of
$x_1$ is proportional to
$$
(1 + |x|^2)^{-{1 \over 2} (\omega + \beta)}, 
$$
thus implying (\ref{dw.3}).

\end{proof}

As in the setting of Section \ref{S2}, we allow for the parameters $\sigma, \nu, \omega$ in
(\ref{dw.1}), (\ref{dw.2}), (\ref{dw.3}) respectively to be different for different rows. 

\subsection{The Gaussian case}
The PDF $p(t)$
in (\ref{m1}) is then the PDF for $S_N=Y_1+\cdots+Y_N$, where each $Y_j$ has distribution
Gamma$[\beta/2,1/2\sigma_j^2]$. An infinite series for this PDF, with coefficients defined recursively,
is given in \cite{Si92}. A closed form expression is available in the special case that
\begin{equation}\label{sb}
1/2 \sigma_1^2 = \cdots = 1/2 \sigma_{N_0}^2 = b_1, \quad1/2 \sigma_{N_0+1}^2 = \cdots = 1/2 \sigma_{N}^2 = b_2, 
\end{equation}
since then $S_N \mathop{=}\limits^{\rm d} X_1 + X_2$, where $X_1$, $X_2$ have distribution
Gamma$[\beta N_0/2,b_1]$, Gamma$[\beta (N-N_0)/2,b_2]$
respectively. From this characterisation, it is straightforward to derive (see e.g.~\cite{Ba13}) that
\begin{equation}\label{3.5}
p(t) = {b_1^{\beta N_0/2} b_2^{\beta(N-N_0)/2} \over \Gamma(\beta N /2)} t^{\beta N/2 -1} e^{-b_2 t}
\, {}_1F_1(\beta N_0/2,\beta N/2;(b_2 - b_1)t).
\end{equation}
Use of computer algebra gives
\begin{equation}\label{3.6}
\int_0^\infty t^s p(t) \, dt =  
{\Gamma(\beta N/2 + s) \over \Gamma(\beta N/2)}
b_2^{-s} \, {}_2 F_1\Big (\beta N_0/2,- s; \beta N/2;
1 - {b_2 \over b_1} \Big ).
\end{equation}
In the case $N_0=0$, this reclaims (\ref{2.2}). Also, the limit $b_1 \to \infty$
reclaims (\ref{2.2})
but with $N$ replaced by $N-N_0$. This latter property is consistent with
the effect of taking $b_1 \to \infty$ being to set the first $N_0$ rows equal to zero. Taking the
derivative of (\ref{3.6}) with respect to $s$ and setting $s=0$ gives the evaluation of $\mu_1$
according to (\ref{m1}).

\begin{proposition}
Consider a product of independent Gaussian random matrices (\ref{GK}), with variances given by (\ref{sb}), and rows ordered so that
$b_2 \le b_1$. We have
\begin{multline}\label{3.7}
\mu_1 = {1 \over 2} \bigg (
\Psi(\beta N/2) - \log b_2 +
{\Gamma(\beta N/2) \over \Gamma(\beta N_0/2) \Gamma(\beta(N-N_0)/2)} \\
\times
\int_0^1 x^{\beta N_0/2 -1} (1 - x)^{\beta (N - N_0)/2 - 1}
\log \bigg ( 1 - \Big ( 1 - {b_2 \over b_1} \Big ) x  \bigg ) \, dx \bigg ).
\end{multline}
\end{proposition}

\begin{proof}
This follows from the stated strategy, upon use of the Euler integral form of ${}_2F_1$.
\end{proof}

The integral formula (\ref{3.7}) is well suited to an asymptotic analysis.

\begin{corollary}
Consider (\ref{3.7}). For $N \to \infty$ with $N_0/N$ fixed, $b_2/b_1$ fixed, we have
\begin{equation}\label{a1}
\mu_1 \sim {1 \over 2} \bigg ( \Psi(\beta N/2) - \log b_2 +
\log \Big ( 1 - \Big ( 1 - {b_2 \over b_1}  \Big )  {N_0 \over N}\Big ) \bigg ).
\end{equation}
If in addition $b_1, b_2$ are proportional to $N$, this further simplifies to give
\begin{equation}\label{a2}
\mu_1 \sim
{1 \over 2} \log \Big ( {\beta \over 2} \Big ( {N-N_0 \over b_2} + {N_0 \over b_1} \Big ) \Big ).
\end{equation}
In the case that $N_0 = N- 1$, $b_1$ is proportional to $N$ and $b_2$ is fixed, we have
\begin{equation}\label{a3}
\mu_1 \sim {1 \over 2} \bigg (
\log {\beta N \over 2 b_1}  + {(\beta/2)^{\beta/2} \over \Gamma(\beta / 2)} \int_0^\infty e^{-\beta x/2} x^{\beta/2 - 1} \log \Big ( 1 + {b_1 \over N b_2} x
\Big ) \, dx \bigg ).
\end{equation}
\end{corollary}

\begin{proof}
In the setting of (\ref{a1}), the integrand in (\ref{3.7}) has it maximum value at $x = N_0/N$. Since the
ratio of gamma functions prefactor is just the reciprocal of the same integral without the log term,
(\ref{a1}) follows.  In the latter, using the fact that for large
$x$, $\Psi(x) \sim x + {\rm O}(1/x)$,  gives (\ref{a2}). To obtain (\ref{a3}), we first change
variables $x \mapsto 1 - x$ in (\ref{3.7}). Manipulation of the log term in the integrand shows
\begin{multline*}
\mu_1 = {1 \over 2} \bigg (
\Psi(\beta N/2) - \log b_1 +
{\Gamma(\beta N/2) \over \Gamma(\beta N_0/2) \Gamma(\beta(N-N_0)/2)} \\
\times
\int_0^1 x^{\beta (N - N_0)/2 -1} (1 - x)^{\beta (N_0 /2 - 1}
\log \bigg ( 1 - \Big ( 1 - {b_1 \over b_2} \Big ) x  \bigg ) \, dx \bigg ).
\end{multline*}
Again, the ratio of gamma function prefactor can be written as the same integral without the log term.
Rewriting as such, and setting $N_0 = 1$, the asymptotic form follows by changing variables 
$x \mapsto x/N$, and taking $N$ large.
\end{proof}

\begin{remark}
The matrix $\Sigma$ in (\ref{GK}), in the setting of (\ref{sb}), has for the sum of its diagonal entries
$$
{\rm Tr} \, \Sigma = {N - N_0 \over 2 b_2} + {N_0 \over 2 b_1}.
$$
Thus (\ref{a2}) can be rewritten
\begin{equation}\label{a2a}
\mu_1 \sim {1 \over 2} \log ( \beta {\rm Tr} \, \Sigma).
\end{equation}
This has been derived previously in \cite{Ka14} (the additional factor of $\beta$ in the logarithm on the LHS of (\ref{a2a}) is
due to our use of standard real Gaussians, whereas in \cite{Ka14}, standard complex (quaternion)
Gaussians are used for $\beta = 2$ ($\beta = 4$)). With $S_N$ defined above (\ref{sb}),
another viewpoint on (\ref{a2a}) is that to leading order for large $N$,
\begin{equation}\label{SN}
\langle \log S_N \rangle \sim
\log \langle S_N \rangle.
\end{equation}

For $\beta = 2$ an evaluation easily seen to be equivalent to (\ref{a3}) with $\beta = 2$ is also given in
\cite[Eq.~(2)]{Ka14}. In addition, the latter reference also treats the case
$\beta = 1$, where $2b_1/N = 1$, $2 b_2 = t$, and we read off from Thm.~1.3 therein the
formula
$$
\lim_{N \to \infty} \mu_1 = {1 \over 2} e^{t/2} \int_1^\infty e^{-tx/2} {dx \over \sqrt{x} (\sqrt{x} + 1)}.
$$
Numerical evaluation indicates that our formula is equivalent to this, although we
don't have a direct proof.

\end{remark}

The evaluation (\ref{3.7}) of the largest Lyapunov exponent for products of independent Gaussian random
matrices (\ref{GK}), with variances given by (\ref{sb}), is based on the explicit evaluation (\ref{3.5})
of the PDF $p(t)$ in (\ref{m1}). It is an elementary fact that in the general variance case $\hat{p}(k)$,
the Fourier transform of $p(t)$, can be evaluated explicitly. From this starting point,
and with the use of complex analysis, Kargin \cite{Ka14}
has obtained an integral evaluation of the largest Lyapunov exponent in the case of general variances
$1/2\sigma_i^2 = b_i$,
\begin{align}\label{K1}
2 \mu_1 & = - \gamma + \log 2 +
\int_0^\infty \bigg ( \chi_{x \in (0,1)} - \prod_{l=1}^N \Big ( 1 + {x \over 2 b_l} \Big )^{-\beta/2} 
\bigg ) {dx \over x}, \nonumber \\
& = -\gamma +
\int_0^\infty \bigg ( \chi_{x \in (0,1)} - \prod_{l=1}^N \Big ( 1 + {x \over  b_l} \Big )^{-\beta/2} 
\bigg ) {dx \over x}
\end{align}
where $\gamma$ denotes Euler's constant and $\chi_J = 1$ for $J$ true,
$\chi_J =0$ otherwise (we have added $\log \beta$ to \cite[Eq.~(1)]{Ka14} due to
our use of standard real Gaussians; recall the discussion below (\ref{a2a})). Note that in the special case
(\ref{sb}), (\ref{K1}) remains distinct from (\ref{3.7}). Here we provide a derivation of
(\ref{K1}) that is independent of complex analysis, and thus distinct from that in \cite{Ka14}.

\noindent
{\it Proof of (\ref{K1}).} \: 
In \cite{ZD88}, the exact formula (\ref{Z}) below was given for the
largest Lyapunov exponent in the case of
the product of independent matrices of the form $\mathbb I_2 + {1 \over c} G$,
where $G$ is a $2 \times 2$ standard Gaussian matrix.
This calculation was based on the particular integral form of the logarithm function
\begin{equation}\label{lg}
\log y = \int_0^\infty {e^{-t} - e^{-ty} \over t} \, dt, \qquad y > 0.
 \end{equation}
 
 Applying (\ref{lg}) to the present setting of real, complex or quaternion standard Gaussian
 entries along each row, with all independent parts distributed as N$[0,\sqrt{1/(2b_l)}]$
 (i.e.~a mean zero normal with standard deviation $\sqrt{1/(2b_l)}$), shows
\begin{equation}\label{1.an} 
2 \mu_1 = \int_0^\infty \Big ( e^{-t} - \prod_{l=1}^N \prod_{s=1}^\beta \langle e^{-t (x_l^{(s)})^2}
\rangle_{x_l^{(s)} \in {\rm N}[0,\sqrt{1/(2b_l)}]} \Big ) \, {dt \over t}.
\end{equation}
Computing the average is elementary, reducing (\ref{1.an}) to
\begin{equation}\label{1.ana}
2 \mu_1 = \int_0^\infty \Big ( e^{-t} - \prod_{l=1}^N \Big ( 1 + {t \over b_l} \Big )^{-\beta/2} \Big )
\, {dt \over t}.
\end{equation}
Upon minor manipulation, this reclaims (\ref{K1}). \hfill $\square$

\subsection{Beta type I distribution}
Here we consider the case that each row of the random matrix has beta type I distribution (\ref{dw.2}) with $\nu \mapsto \nu_l$,  appropriately generalised to allow for complex or quaternion entries.
An evaluation of
 the Lyapunov exponent as an infinite series can be given.

\begin{proposition}\label{p10}
In the above specified setting,
\begin{equation}\label{pq1}
\mu_1 ={1 \over 2} \sum_{n=-\infty}^\infty
c_n \, \int_0^{\beta N} (\log x) e^{2 \pi i n x/\beta N} \, dx,
\end{equation}
where
\begin{equation}\label{pq2}
c_n = {1 \over \beta N} \prod_{l=1}^N 
 \, {}_1 F_1(\beta/2,\beta/2+\alpha_l;-2 \pi i n/\beta N) ,
\end{equation}
with $\alpha_l = \beta (N + \nu_l - 1)/2 $.
\end{proposition}

\begin{proof}
Generally, in the setting that each $x_{l}$ has PDF $q_l(x)$, the Fourier transform
$\hat{p}(k)$ of the PDF for the distribution of $S_N := \sum_{l=1}^N 
|x_{l}|^2$ is, according to the convolution formula, given by
\begin{equation}\label{pq}
\hat{p}(k) = \prod_{l=1}^N  \int_{\mathbb R^\beta} e^{i k |x|^2} q_l(x) \, dx^{(1)}
\cdots d x^{(\beta)} .
\end{equation}
For $q_l(x)$ equal to (\ref{dw.2}) with $\nu \mapsto \nu_l$, $q_l(x)$ is supported on
$|x| < 1$ and so $p(x)$ is supported on $(0,\beta N)$. Thus for $x \in (0,\beta N)$ we have the
Fourier series
$$
p(x) = \sum_{n = - \infty}^\infty c_n e^{2 \pi i n x/\beta N}, \qquad
c_n = {1 \over \beta N} \int_0^{\beta N} p(x) e^{-2 \pi i x n/\beta N} \, dx.
$$
Making use of (\ref{pq}) it follows that
$$
c_n = {1 \over \beta N} \prod_{l=1}^N 
 \int_{|x| < 1} e^{-2 \pi i |x|^2 n/\beta N} q_l(x) \, dx^{(1)}
\cdots d x^{(\beta)} .
$$
Substituting (\ref{dw.2}) for $q_l(x)$, and changing variables to polar coordinates,
an integral
representation of the confluent hypergeometric function results, and
shows the integral can be evaluated to give (\ref{pq2}).
Recalling (\ref{m1}), (\ref{pq1}) follows.
\end{proof}

Use of the integral form of the logarithm (\ref{lg}) allows for an alternative 
expression for the Lyapunov exponent, which is the analogue of (\ref{1.ana}).

\begin{proposition}
An alternative expression to the result of Proposition \ref{p10} is the formula
\begin{equation}\label{0s}
2 \mu_1 =
\int_0^\infty \bigg ( e^{-t}-
\prod_{l=1}^N   {}_1F_1(\beta/2,\beta/2+\alpha_l;-t) 
\bigg ) {dt \over t}.
\end{equation}
\end{proposition}

\begin{proof}
The essential point is that for $p(x)$ given by (\ref{dw.2}),
\begin{equation}\label{os}
\langle e^{-t |x|^2} \rangle_{x \in p(x)} = \int_{|x| < 1} 
e^{-t |x|^2} p(x) \, dx^{(1)} \cdots d x^{(\beta)} =
 {}_1F_1(\beta/2,\beta/2+\alpha_l;-t) .
\end{equation}
This follows from the use of polar coordinates, and recognition of the resulting
expression as a standard integral form of ${}_1 F_1$.
\end{proof}

We note from (\ref{CN}) with $\beta = 1$ that the choices $(N,\nu) = (2,1)$, and
$(N,\nu) = (3,0)$, give the marginal distribution of each entry as uniform on $(-1,1)$.
Furthermore, entries in different rows are independent, so the PDF $p(t)$ in
(\ref{m1}) is then the PDF for $\sum_{i=1}^N U_i^2$, where each $U_i$ is an
independent random variable uniformly distributed on $(0,1)$. For general $N$, this
random sum has been the topic of the recent works \cite{We17,Fo18}, and for $N=2$
arose in the recent study \cite{FZ18}.

\subsection{Beta type II distribution}

As for the Gaussian and beta type I cases, the identity (\ref{lg}) again quickly leads
to an explicit integral formula for the largest Lyapunov exponent in the
case of beta type II distribution.

\begin{proposition}\label{P9}
Specify that each row of the random matrix has  beta type II distribution (\ref{dw.3}) with $\omega \mapsto \omega_l$,  appropriately generalised to allow for complex or quaternion entries. The corresponding
Lyapunov exponent is given by
\begin{equation}\label{1s}
2 \mu_1 = 
\int_0^\infty \bigg ( e^{-t} -
\prod_{l=1}^N \Big ( {\Gamma(\omega_l + \beta/2) \over \Gamma(\omega_l/2)} 
U(\beta/2,1-\omega_l/2;t)  \bigg ) {dt \over t},
\end{equation}
where $U(a,b;x)$ denotes the confluent
hypergeometric function of the second kind.
\end{proposition}

\begin{proof}
Analogous to (\ref{os}), for $p(x)$ given by (\ref{dw.3}),
\begin{equation}\label{os1}
\langle e^{-t |x|^2} \rangle_{x \in p(x)} = \int_{\mathbb R^\beta} 
e^{-t |x|^2} p(x) \, dx^{(1)} \cdots d x^{(\beta)} =  {\Gamma(\omega_l + \beta/2) \over \Gamma(\omega/2)} 
U(\beta/2,1-\omega_l/2;t),
\end{equation}
where again polar coordinates are used, and a standard integral form of $U(a,b;x)$ is recognised
to obtain the final equality.
\end{proof}

\begin{remark}
The structural similarity between (\ref{1s}), (\ref{0s}) and (\ref{K1}) is evident.
In fact the latter can be reclaimed from either of (\ref{1s}) or (\ref{0s}). Consider for definiteness the former.
Set $\omega_l = 2 b_l L$. From the integral form of $U(a,b;x)$ as implicit in
(\ref{os1}),
we see that as $L \to \infty$
$$
 {\Gamma(\omega_l + \beta/2) \over \Gamma(\omega_l/2)} 
U(\beta/2,1-\omega_l/2;x) \to {1 \over (1 + x/b_l)^{\beta/2}}
$$
and thus $2 \mu_1 + \log L$ tends to (\ref{K1}).

\end{remark}

\begin{remark}
There is an analogue of the large $N$ result (\ref{a2a}) in both the type I and type II beta cases,
$$
\mu_1 \sim {1 \over 2} \log \sum_{j=1}^N (1 + \nu_j)^{-1}, \qquad
\mu_1 \sim {1 \over 2} \log \sum_{j=1}^N (\omega_j - 2/\beta)^{-1},
$$
respectively. This is a consequence of (\ref{SN}).
\end{remark}

\begin{remark}
As a part of Remark \ref{R4}, methods to sample an element in the first column of $X$, $x_{l,1}$
say, in each of the 3 cases under consideration --- rows of $X$ sampled according to the PDFs
and their extension to their extension to the complex and quaterion cases --- was given.
This then allows for a Monte Carlo approximation to (\ref{m1}) according to
$$
\mu_1 \approx {1 \over 2 M} \sum_{j=1}^M \log \Big ( \sum_{l=1}^N | x_{l,1}^{(j)} |^2 \Big ),
$$
where $ x_{l,1}^{(j)}$ denotes the $j$-th sample of $x_{l,1}$ and $M$ is the total number of
times the column is sampled. This allows for (\ref{1.ana}), (\ref{0s}) and (\ref{1s}) to be checked
numerically; agreement is found although the integrand in (\ref{0s}) typically requires high precision
evaluation to obtain numerical stability.
\end{remark}

%\subsection{Alternative integral formulas for $\mu_1$}
%Let $G$ be a $2 \times 2$ standard Gaussian matrix, and consider a product of independent
%matrices of the form
%$\mathbb I_2 + {1 \over c} G$. In \cite{ZD88} the Lyapunov exponent was calculated to be given by
%(\ref{Z}) below, by making use of the particular integral form of $\log y$, $y>0$,
%\begin{equation}\label{lg}
%\log y = \int_0^\infty {e^{-t} - e^{-ty} \over t} \, dt.
% \end{equation}
%Use of this identity allows us to compute alternative integral formulas for the largest Lyapunov
%exponent of the product ensembles considered in the above section.

%\subsection{The variance}
%Let $\mathbf x_0$ be a column vector of size $N \times 1$, and consider the random product matrix
%(\ref{PX}). In keeping with (\ref{6}), the multiplicative ergodic theorem of Oseledec \cite{Os68}
%gives that for almost all $\mathbf x_0$, $\mu_1 = \lim_{m \to \infty} || P_m \mathbf x ||$. Hence for
%$m$ large but finite $ || P_m \mathbf x || \sim e^{m \mu_1}$. In the case that each matrix
%$X_i$ in the product has the invariance (\ref{4}), we know from \cite[Prop.~2]{CN84} that
%$y_m := {1 \over m} \log || P_m \mathbf x ||$ has a Gaussian distribution
%${\rm N}[\mu_1,\tilde{\sigma}_1^2/m]$, $\tilde{\sigma}_1^2 = \lim_{m \to \infty} m {\rm Var} \,
%y_m$. Following on from earlier work on products of standard Gaussian random matrices
%\cite{CN84,ABK14,Ip14}, for the general variance case (\ref{GK}) and real, complex or quaternion
%entries, an expression involving the multidimensional integral (\ref{K1})
%was obtained in \cite{Fo15}

\section{Noncentral Wishart matrices}
In \S \ref{S2.2} the random matrix $Y = cI_N + X$, where $X$ has i.i.d.~real Gaussian entries
N$[0,\sigma]$, was introduced, as were the complex and quaternion analogues. It has been
observed in \cite{ZD88} that the invariance (\ref{4a}) implies the simple formula (\ref{6}) ---
with the matrix $X$ replaced by $Y$ --- for the sum of the Lyapunov exponents. Thus in the case
$k=N$, for $c$ large the asymptotic formula (\ref{2.28}) holds true. A similar strategy
allows for the derivation of the corresponding result in the general $k$ case.

Let $Y_k$ denote first $k$ columns of the matrix $Y$.
The matrix $W_k:=Y_k^\dagger Y_k$ is, as for the case $k=N$, an example of a non-central Wishart
matrix. From the work of Constantine \cite{Co63},
generalised to the complex case in \cite[Eq.~(21)]{CZ05}, with $\tilde{c} = c^2/2 \sigma^2$ we have
\begin{align}\label{FW1}
\Big \langle (\det W_k)^\alpha \Big \rangle & =
\prod_{l=1}^k (2 \sigma^2)^\alpha (\beta (N-l+1)/2)_\alpha \,
e^{-k \tilde{c}} \, {}_1^{} F_1^{(2/\beta)}(\alpha + \beta N/2;\beta N/2;\tilde{c}I_k) \nonumber \\
& = \prod_{l=1}^k (2 \sigma^2)^\alpha (\beta(N-l+1)/2)_\alpha \,
{}_1^{} F_1^{(2/\beta)}(-\alpha,\beta N/2;-\tilde{c}I_k).
\end{align}
More revealing is the simplification of (\ref{FW1}) obtained upon
substitution of the large $\tilde{c}$ asymptotic expansion
 \cite[Th.~3.2]{CM76}, \cite[\S 7]{Mu78} 
 for ${}_1^{} F_1^{(2/\beta)}$ in the case $\beta = 1$ known from
  \cite[Th.~3.2]{CM76}, \cite[\S 7]{Mu78}, appropriately generalised to all
 $\beta > 0$ as noted above (\ref{2.26}). This gives as a generalisation of (\ref{2.26})
 \begin{equation}\label{2.26a}
\Big \langle (\det W_k)^\alpha \Big \rangle
 \sim {}_2^{} F_0^{(2/\beta)} \Big ( 1  - \alpha - \beta(N-k+1)/2, - \alpha; {1 \over \tilde{c}}  \mathbb I_k\Big ).
 \end{equation}
 Differentiating with respect to $\alpha$ and setting $\alpha = 0$, it then follows that
 as a generalisation of (\ref{2.28}),  to leading
 order
  \begin{equation}\label{2.28a}
\mu_1 + \cdots + \mu_k = {1 \over 2} \Big \langle
 \log   \det W_k  \Big \rangle
 \mathop{\sim}\limits_{c \to \infty} {\sigma^2 k (\beta(N-k+1)/2 -1) \over c^2}.
  \end{equation}
Notice that in the case $k=N-1$, $\beta = 1$ the coefficient on the RHS vanishes.  Moreover (\ref{2.26a}) shows
that each term in the asymptotic expansion vanishes, indicating that the large $c$ decay is in
fact exponentially fast. In the particular case of this type, $\beta = 1$, $N=2$, $k=1$ it was shown in
\cite[Eq.~(72)]{ZD88} that
 \begin{equation}\label{Z}
\mu_1 =  
 {1 \over 2} \int_{c^2/2\sigma^2}^\infty {e^{-u} \over u} \, du,
\end{equation}
which decays as a Gaussian at infinity.

Setting $c=0$ in (\ref{FW1}) shows that for $X$ a Gaussian matrix with real, complex or
quaternion entries, with parts of each entry distributed as N$[0,\sigma]$,
$$
\Big \langle \det (X^\dagger_{N \times k} X_{N \times k} ) \Big \rangle =
\prod_{l=1}^k (2 \sigma^2)^\alpha (\beta(N-l+1)/2)_\alpha.
$$
This allows the RHS of (\ref{6}) to be computed, telling us that
$$
\mu_1 + \cdots + \mu_k = {1 \over 2} \Big ( \log (2 \sigma^2) +
\sum_{l=1}^k \Psi(\beta(N-l+1)/2) \Big )
$$
(cf.~(\ref{cw.1})), as was known previously \cite{Ne86}, \cite{Fo12}.
Generalisations relating to quantifying the convergence of this Lyapunov spectrum
as a Gaussian random variable with particular variances are given in
\cite{ABK14,Ip14,Fo15}

For $k=1$ and general $N$, $\beta$, use can be made of (\ref{m1}) to get a formula for $\mu_1$ which is distinct
from that implied by (\ref{2.26a}). In the present setting the function $p(t)$ in (\ref{m1}) is the PDF for
$S_N = (\sigma/c)^2 \sum_{j=1}^N \sum_{s = 1}^\beta (x_{1j}^{(s)})^2$, where $x_{11}^{(1)}$, 
is distributed as N$[c/\sigma,1]$ and $x_{1j}^{(s)}$, $s=1,\dots,\beta$, $j=1,\dots,N$, $(s,j) \ne (1,1)$
is distributed as N$[0,1]$. After scaling out the factor of $(\sigma/c)^2$, this sum specifies the noncentral
$\chi$ squared distribution with $\beta N$ degrees of freedom and noncentrality parameter
$\lambda = (c/\sigma)^2$. From the known PDF for the latter
(see e.g.~\cite{Wiki}), we thus have
  \begin{equation}\label{2.29a}
  \mu_1 = {1 \over 2} \int_0^\infty \log  (  t/ \lambda  ) f(t;\beta N,\lambda) \, dt ,
   \end{equation} 
   where
  \begin{equation}\label{2.30a}  
    f(t;k,\lambda) = {1 \over 2} e^{-(t+\lambda)/2} \Big ( {t \over \lambda}  \Big )^{k/4 - 1/2} I_{k/2 - 1} (\sqrt{\lambda t}).
  \end{equation} 
For large $\lambda$, use of the  leading asymptotic expansion
$
I_\nu(z) \sim  e^z / (2 \pi z)^{1/2} ,
$
and expansion of the integrand about its maximum at $t = \lambda$ reclaims (\ref{2.28a}) in the
case $k=1$. The form of (\ref{2.29a}) with $\beta = 1$, $N=2$ is different to (\ref{Z}). Both are simple to compute numerically and give the same value.

%\cite{Si92} that the PDF for $S_N$ can be written in the form of an infinite series
%\begin{equation}\label{3.5}
%p(x) = {1 \over \Gamma(a_N)} \Big ( \prod_{j=1}^N \lambda_j^\beta \Big ) x^{a_N - 1} e^{-\lambda_N x}
%\sum_{k=0}^\infty {b_N(k) (\beta (N - 1)/2)_k \over (\beta N /2)_k k!}
%((\lambda_N - \lambda_{N-1}) x )^k,
%\end{equation}
%where
%\begin{align*}
%\lambda_j & = 1/2 \sigma_j^2 \\
%b_i(k) & = \left \{
%\begin{array}{ll} 1, & i=2 \\
%\displaystyle \sum_{j=0}^k {b_{i-1}(j) (\beta(i-2))_j (-k)_j \over
%(\beta(i-1)/2)_j j!} c_i^j, & i=3,\dots,N \end{array} \right.
%\\
%c_i & = (\lambda_{i-2} - \lambda_{i-1})/(\lambda_i - \lambda_{i-1}).
%\end{align*}

\section*{Acknowledgements}
This research project is part of the program of study supported by the 
ARC Centre of Excellence for Mathematical \& Statistical Frontiers,
and the Australian Research Council Discovery Project grant DP170102028.
The work of JZ was supported by a University of Melbourne Research Scholarship.

\providecommand{\bysame}{\leavevmode\hbox to3em{\hrulefill}\thinspace}
\providecommand{\MR}{\relax\ifhmode\unskip\space\fi MR }
\providecommand{\MRhref}[2]{%
  \href{http://www.ams.org/mathscinet-getitem?mr=#1}{#2}
}
\providecommand{\href}[2]{#2}

\end{document}